\documentclass{article}

\usepackage{vmargin,epsfig}
\usepackage[francais]{babel}
\usepackage[latin1]{inputenc}
\usepackage[tbtags]{amsmath}
\usepackage{amsthm}
\usepackage{amssymb}
\usepackage{pifont}
\usepackage[all]{xy}
\usepackage{wrapfig}
\usepackage{calc}

\makeatletter
\renewcommand{\everyentry@}{\vphantom{\hat \M_1}}
\makeatother

\newtheorem{theo}{Théorème}
\newtheorem{lemme}[theo]{Lemme}
\newtheorem{prop}[theo]{Proposition}
\newtheorem{cor}[theo]{Corollaire}
\theoremstyle{definition}

\def\leq{\leqslant}
\def\geq{\geqslant}

\def\Z{\mathbb{Z}}

\def\F{\mathbb{F}}

\def\pa#1{\left(#1\right)}

\def\epsilon{\varepsilon}

\def\calB{\mathcal{B}}
\def\calR{\mathcal{R}}
\def\calN{\mathcal{N}}

\def\id{\text{\rm id}}

\def\Mp{\text{\rm Mod}^\phi_{/K}}

\title{Sur la classification de quelques $\phi$-modules simples}
\author{Xavier Caruso}
\date{Juillet 2008}

\begin{document}

\maketitle

Dans cet appendice, on détermine les $\phi$-modules étales simples sur
$\bar F_p((u))$ dans une situation légèrement plus générale que celle
étudiée dans l'article (\cite{hellmann}). On fixe $p$ un nombre premier
et on pose $k = \bar \F_p$ et pour tout $q$, puissance de $p$, on
définit $\F_q$ comme l'unique sous-corps de $k$ de cardinal $q$. Soit
$\sigma$ un automorphisme de $k$.et $b > 1$ un entier. On note $\ell$ le
sous-corps de $k$ fixe par $\sigma$ et, plus généralement, pour tout
entier $d$, on note $\ell_d$ le sous-corps fixe par $\sigma^d$. On
considère le corps $K = k((u))$ que l'on munit de l'endomorphisme :
$$\phi\pa{\sum_{n=-\infty}^{\infty} a_n u^n} =
\sum_{n=-\infty} ^{\infty} \sigma(a_n) u^{bn}.$$
(Si $b = p$ et $\sigma = \id$, on retrouve donc la situation de
l'article.)

On considère la catégorie $\Mp$ dont les objets sur les $K$-espaces
vectoriels $D$ de dimension finie muni d'un endomorphisme
$\phi$-semi-linéaire $\phi_D : D \to D$ dont l'image contient une base
de $D$. Par la suite, lorsque cela ne prêtera pas à confusion, nous
noterons simplement $\phi$ à la place de $\phi_D$. Les morphismes de
$\Mp$ sont bien entendu les applications $K$-linéaires qui commutent à
$\phi$. On vérifie aisément que la catégorie $\Mp$ est abélienne et que
chaque objet est de longueur finie. Le but de cette appendice est de
déterminer les objets simples de $\Mp$.

\subsection*{Les objets $D(d,n,a)$ et $D(r,a)$}

Soient $d$ un entier strictement positif, $n$ un entier naturel et $a$
un élément non nul de $k$. À ces données, on associe un objet de $\Mp$
noté $D(d,n,a)$ défini comme suit :
\begin{itemize}
\item[$\bullet$] $D(d,n) = K e_0 \oplus K e_1 \oplus \cdots 
\oplus K e_{d-1}$ ;
\item[$\bullet$] $\phi(e_i) = e_{i+1}$ pour $i \in \{0, \ldots, d-2\}$.
\item[$\bullet$] $\phi(e_{d-1}) = a u^n e_0$.
\end{itemize}
On définit également $D(d,n) = D(d,n,1)$.

\bigskip

On commence par déterminer des familles d'isomorphismes entre les
différents $D(d,n,a)$.

\begin{lemme}
\label{lem:a1}
Si $\sigma$ n'est pas l'identité, alors $D(d,n,a)$ est isomorphe à
$D(d,n)$ pour tous $d$, $n$ et $a$ comme précédemment.
\end{lemme}

\begin{proof}
Considérons un élément $\lambda \in k$ tel que $a = \frac {\sigma^d 
(\lambda)} {\lambda}$ (l'existence résulte d'un théorème classique de
théorie de Galois). L'isomorphisme $D(d,n,a) \to D(d,n)$ est alors
donné par $e_i \mapsto \sigma^i (\lambda) e_i$.
\end{proof}

\begin{lemme}
\label{lem:isom}
Soient $(d,n,a)$ et $(d,n',a')$ deux triplets comme précédemment avec le
même $d$. On suppose qu'il existe un entier naturel $s$ tel que $n
\equiv b^s n' \pmod {b^d-1}$ et $a = \sigma^s(a')$. Alors $D(d,n,a) 
\simeq D(d,n',a')$.
\end{lemme}

\begin{proof}
Si $m = \frac{n - b^s n'}{b^d-1} \in \Z$, et si $e_0, \ldots, e_{d-1}$
(resp. $e'_0, \ldots, e'_{d-1}$) est la base fournie par la définition,
alors un isomorphisme est $e_i \mapsto u^{b^i m} e'_{i+s}$ où les $e'_j$
pour $j \geq d$ sont définis par la relation de récurrence $e'_{j+1} =
\phi(e'_j)$.
\end{proof}

\begin{prop}
\label{prop:decomp}
Soient $(d,n,a)$ un triplet comme précédemment. On suppose qu'il existe
$d'$ et $n'$ comme précédemment tels que $t = \frac d{d'}$ soit un entier
et que $\frac n{b^d-1} = \frac{n'}{b^{d'}-1}$.
\begin{itemize}
\item[(i)] Si $\sigma$ n'est pas l'identité, on a :
$$D(d,n,a) \simeq D(d',n')^{\oplus \frac d {d'}}.$$
\item[(ii)] Si $\sigma$ est l'identité et si $t$ est premier avec $p$, on a :
$$D(d,n,a) \simeq D(d',n',a'_1) \oplus D(d',n',a'_2) \oplus \cdots
\oplus \simeq D(d',n',a'_t)$$
où les $a'_i$ sont les racines $t$-ièmes de $a$.
\item[(iii)] Si $\sigma$ est l'identité, $a=1$ et $t = p$, il existe une
suite croissantes de sous-modules de $D(d,n)$ stables par $\phi$
$$0 = D_0 \subset D_1 \subset D_2 \subset \cdots \subset D_p = D(d,n)$$
pour laquelle tous les quotients $D_m/D_{m-1}$ ($1 \leq m \leq p$) sont
isomorphes à $D(d',n')$.
\end{itemize}
\end{prop}

\begin{proof}
Pour tout la preuve, posons $r = \frac n{b^d-1} = \frac{n'}{b^{d'}-1}$.

On traite d'abord (i). D'après le lemme \ref{lem:a1}, on peut supposer $a
= 1$. Étant donné que le groupe de Galois absolu de $\F_p$ est un groupe
procyclique sans torsion (il est isomorphe à $\hat \Z$), $\ell_d$ est
une extension cyclique de degré $t$ de $\ell_{d'}$. On définit pour tout
$\alpha \in k$ et $i \in \{0, \ldots, d'-1\}$, les éléments suivants de
$D(d,n,a)$ :
$$f_i(\alpha) = \sum_{s=0}^{t-1} \sigma^{sd'+i}(\alpha) \: 
u^{-rb^i(b^{sd'}-1)} \: e_{sd'+i}.$$
(On remarquera que l'exposant qui apparaît sur $u$ est un entier étant
donné $r (b^{d'}-1) = n'$ en est un.) On vérifie à la main que
$\phi(f_i(\alpha)) = f_{i+1} (\alpha)$ pour $i \in \{0, \ldots, d'-2\}$, 
et que $\phi(f_{d'-1}(\alpha)) = u^{n'} f_0(\alpha)$ (la dernière égalité 
utilise $\sigma^d(\alpha) = \alpha$, ce qui est vrai puisque $\alpha$ est 
pris dans $\ell_d$). Comme par ailleurs, il est clair que à $\alpha$ fixé, 
la famille des $f_i(\alpha)$ est libre, on en déduit que les sous-objets
$F(\alpha) = K f_0(\alpha) \oplus \cdots \oplus K f_{d-1}(\alpha)$ sont 
tous isomorphes à $D(d',n')$. Il suffit donc pour conclure de montrer que
$t$ d'entre eux sont en somme directe. Ceci nous amène à chercher des
éléments $\alpha_1, \ldots, \alpha_t$ tels que chacune des matrices :
$$M_i = \pa{ \begin{array}{cccc}
\sigma^i(\alpha_1) & \sigma^i(\alpha_2) & \cdots & \sigma^i(\alpha_t) \\
\sigma^{d'+i}(\alpha_1) & \sigma^{d'+i}(\alpha_2) & \cdots & 
\sigma^{d'+i}(\alpha_t) \\
\sigma^{2d'+i}(\alpha_1) & \sigma^{2d'+i}(\alpha_2) & \cdots & 
\sigma^{2d'+i}(\alpha_t) \\
\vdots & \vdots & & \vdots \\
\sigma^{(t-1)d'+i}(\alpha_1) & \sigma^{(t-1)d'+i}(\alpha_2) & \cdots & 
\sigma^{(t-1)d'+i}(\alpha_t) \\
\end{array} }$$
(pour $i$ variant entre $0$ et $d'-1$) soit inversible. En réalité, il
suffit pour cela de choisir les $\alpha_i$ de façon à ce que $(a_1, \ldots,
a_t)$ soit une base de $\ell_d$ sur $\ell_{d'}$. En effet, on
remarque d'abord que $M_i = \sigma^i(M_0)$ et donc qu'il suffit de
démontrer l'inversibilité de $M_0$. On invoque alors le théorème d'Artin
d'indépendance linéaire des caractères qui montre que les vecteurs ligne
forment une famille libre.

Passons maintenant au (ii) : on suppose donc $\sigma = \id$. On pose
cette fois-ci :
$$f_i(\alpha) = \sum_{s=0}^{t-1} \alpha^{t-1-s} \: 
u^{-rb^i(b^{sd'}-1)} \: e_{sd'+i}$$
pour tout $i \in \{0, \ldots, d'-1\}$ et tout $\alpha \in k$.
On a encore $\phi(f_i(\alpha)) = f_{i+1} (\alpha)$ pour $i \in \{0,
\ldots, d'-2\}$, et si $\alpha$ est une racine $t$-ième de $a$, on
vérifie directement que $\phi(f_{d-1}(\alpha)) = \alpha u^{n'} f_0
(\alpha)$. Ainsi, comme il est clair par ailleurs que la famille des
$f_i(\alpha)$ ($0 \leq i < d'$) est libre, on a $K f_0(\alpha) + \cdots
+ K f_{d-1}(\alpha) \simeq D(d',n',\alpha)$ (toujours sous l'hypothèse
$\alpha^t = a$).
Si maintenant $t$ est premier à $p$, $a$ admet $t$ racines $t$-ièmes
distinctes, et un calcul de déterminants de Vandermonde montre
facilement que la famille des $(f_i(\alpha))_{0 \leq i < d', \, \alpha^t
= a}$ est une base de $D(d,n,a)$. La conclusion s'ensuit.

Terminons finalement par la démonstration de l'assertion (iii). Pour tous
entiers $i \in \{0, \ldots, d'-1\}$ et $j \in \{0, \ldots, p-1\}$, on 
pose :
$$f_{i,j} = \sum_{s=0}^{p-1} s^j \: u^{-rb^i(b^{sd'}-1)} \: e_{sd'+i}$$
et on définit $D_m$ comme le sous-$K$-espace vectoriel de $D(d,n)$
engendré par les $f_{i,j}$ avec $0 \leq i < d'$ et $0 \leq j < m$. 
À nouveau l'utilisation des déterminants de Vandermonde assure la 
liberté de la famille des $f_{i,j}$, ce qui montre que la dimension
de $D_m$ sur $K$ est $md'$. Par ailleurs, on a les relations 
$\phi(f_{i,m}) = f_{i+1,m}$ pour $i \in \{0, \ldots, d'-2\}$ et :
$$\phi(f_{d-1,m}) = u^{n'} \sum_{\mu=0}^m (-1)^{m-\mu} \binom m \mu 
f_{0,\mu} \equiv u^{n'} f_{0,m} \pmod {D_{m-1}}.$$
Elles montrent à la fois que $D_m$ est stable par $\phi$ pour tout $m$
et que les quotients $D_m/D_{m-1}$ sont tous isomorphes à $D(d',n')$.
\end{proof}

La proposition précédente nous conduist à considérer $\calR_b$
l'ensemble quotient de $\Z_{(b)}$ (le localisé de $\Z$ en la partie
multiplicative des entiers premiers avec $b$) par la relation
d'équivalence :
$$x \sim y \quad \Longleftrightarrow \quad (\exists s \in \Z) \,\,\,
x \equiv b^s y \pmod \Z.$$
\emph{Via} l'écriture en base $b$, les éléments de $\calR_b$
s'interprètent aussi comme l'ensemble des suites périodiques (depuis le
début) d'éléments de $\{0, 1, \ldots, b-1\}$ modulo décalage des
indices, où l'on a en outre identifié les suites constantes égale à $0$ 
et à $b-1$. 

Soit $r \in \calR_d$. Par définition, il est représenté par une fraction
(que l'on peut supposer --- et que l'on supposera par la suite ---
irréductible) $\frac s t$ où $t$ est un nombre premier avec $b$. On
vérifie directement que l'ordre de $b$ modulo $t$ ne dépend pas du
représentant (irréductible) choisi : on l'appelle la \emph{longueur} de
$r$ et on le note $\ell(r)$. On pourra remarquer qu'à travers le point
de vue \og suites périodiques \fg, $\ell(r)$ s'interprète simplement
comme la plus petite période.

Notons $\calN(r)$ l'ensemble des entiers relatifs $n$ pour lesquels
$\frac n {b^{\ell(r)} -1}$ est un représentant de $r$. La définition de
$\ell(r)$ implique immédiatement la non-vacuité de $\calN(r)$. Le lemme
\ref{lem:isom} montre que les objets $D(\ell(r),n)$ pour $n$ variant
dans $\calN(r)$ sont isomorphes entre eux. Notons $D(r)$ l'un de ces
objets. Si en outre $\sigma = \id$, le même lemme \ref{lem:isom} permet
de définir $D(r,a)$ pour tout $a \in k^\star$ comme l'un des 
$\phi$-modules $D(\ell(r),n,a)$, $n \in \calN(r)$.

\begin{theo}
\label{theo:Dr}
Les $D(r)$ (resp. $D(r,a)$ si $\sigma = \id$) sont des objets simples de
$\Mp$. De plus, ils sont deux à deux non isomorphes.
\end{theo}

\begin{proof}
Pour simplifier la preuve, on suppose dans la suite $a=1$, laissant au 
lecteur l'exercice (facile) d'adapter les arguments au cas général.

Notons $\ell = \ell(r)$ et considérons $n$ tel que $\frac n {p^\ell-1}$
représente $r$. Soit $D$ un sous-objet non nul de $D(r)$, et soit $x =
\lambda_1 e_1 + \cdots + \lambda_\ell e_\ell$ un élément non nul de $D$
pour lequel le nombre de $\lambda_i$ non nuls est minimal. Quitte à
remplacer $x$ par $\phi^m(x)$ et pour un certain entier $m$, on peut
supposer $\lambda_1 \neq 0$. Quitte à renormaliser $x$, on peut en outre
supposer $\lambda_1 = 1$. On a alors :
$$\phi^d(x) - u^n x = \sum_{i=2}^\ell (\phi^d(\lambda_i) u^{b^i n} - 
\lambda_i u^n) e_i \in D.$$
Supposons par l'absurde qu'il existe un indice $i > 1$ tel que
$\lambda_i \neq 0$. Montrons dans un premier temps que
$\phi^d(\lambda_i) u^{b^i n} \neq \lambda_i
u^n$. Encore par l'absurde : si ce n'était pas le cas, on déduirait
$\frac{\phi^d(\lambda_i)} {\lambda_i} = u^{-n(b^i-1)}$ et puis $v
(b^d-1) = -n (b^i -1)$ où $v$ est la valuation $u$-adique de
$\lambda_i$. Ainsi, on aurait $r = \frac{-v}{b^i - 1}$, et on
obtiendrait une contradiction avec la définition de $\ell$. Au final,
$\phi^d(\lambda_i) u^{b^i n} \neq \lambda_i u^n$, et
l'élément $\phi^d(x)$ est un élément non nul de $D$ qui, sur la base des
$(e_i)$, a strictement moins de coefficients non nuls que n'en avait
$x$. Ceci contredit la minimalité supposée et montre que $x = e_1$.
On en déduit $e_1$ est élément de $D$ et, puisque ce dernier est par
hypothèse stable par $\phi$, il contient nécessairement tous les $e_i$.
En conclusion, $D = D(r)$ et le théorème est démontré.

Pour prouver que $D(r)$ n'est pas isomorphe à $D(r')$, il suffit de
remarquer que $\ell(r)$ et $\calN(r)$ se retrouvent tous deux à partir 
de $D(r)$ : le premier en est la dimension, alors que le second est
l'ensemble des entiers $n$ pour lesquels il existe un $x \in D(r)$ non
nul vérifiant $\phi^{\ell(r)} (x) = u^n x$. Finalement, il est clair 
qu'à partir de ces deux données, $r$ est entièrement déterminé dans le
quotient $\calR_b$.
\end{proof}

\subsection*{Classification des objets simples}

Nous souhaitons désormais montrer la réciproque du théorème
\ref{theo:Dr}, c'est-à-dire que les objets simples de $\Mp$ sont tous
isomorphes à un certain $D(r)$ (ou $D(r,a)$ si $\sigma = \id$). Pour
cela, nous considérons $D$ un objet simple de $\Mp$, $(e_1, \ldots,
e_d)$ une base de $D$ et $G$ la matrice de $\phi$ dans cette base,
\emph{i.e.} l'unique matrice vérifiant l'égalité :
$$(\phi(e_1), \ldots, \phi(e_d)) = (e_1, \ldots, e_d) G.$$
On rappelle, à ce propos, la formule de changement de base qui
interviendra plusieurs fois dans la suite : si $\calB' = (e'_1, \ldots,
e'_d)$ est une autre base de $D$ et si $P$ est la matrice de passage de
$(e_1, \ldots, e_d)$ à $\calB'$, alors la matrice de $\phi$ dans la base
$\calB$ est donnée par la formule $P^{-1} G \phi(P)$. Il résulte de
cette formule que, quitte à multiplier les $e_i$ par une certaine
puissance de $u$, on peut supposer que $G$ est à coefficients dans
$k[[u]]$. En réalité on aura besoin d'un résultat un peu plus précis,
conséquence du lemme suivant. Introduisons avant tout une dernière
notation : soit $\gamma$ la valuation $u$-adique de $\det G$.

\begin{lemme}
Soit $N$ un entier strictement supérieur à $\frac{p\gamma}{b-1}$ et $H$
une matrice à coefficients dans $k[[u]]$ congrue à $G$ modulo $u^N$.
Alors, il existe une matrice $P$ à coefficients dans $k[[u]]$ et
inversible dans cet anneau telle que $P G \phi(P)^{-1} = H$.
\end{lemme}

\begin{proof}
On définit une suite de matrices $(P_i)$ (\emph{a priori} à coefficients
dans $K$) par $P_0 = I$ et la formule de récurrence $P_{i+1} = H \phi(P)
G^{-1}$. On a directement $P_1 = H G^{-1}$ d'où on déduit, en
utilisant l'hypothèse de l'énoncé, que $P_1 \equiv I \pmod {u^{N-v}}$,
\emph{i.e.} $P_1 - P_0$ est divisible par $u^{N-v}$. Par ailleurs, pour
tout $i \geq 1$, on a $P_{i+1} - P_i = H \phi(P_i - P_{i-1}) G^{-1}$,
d'où il suit que si $P_i - P_{i-1}$ est divisible par $u^v$, alors
$P_{i+1} - P_i$ est divisible $u^{pv-\gamma}$. Une récurrence immédiate
montre alors que $P_{i+1} - P_i$ est divisible par $u^{v_i}$ où la suite
$(v_i)$ est définie par $v_0 = N - \gamma$ et $v_{i+1} = b v_i -
\gamma$.  De $v_0 > \frac \gamma {b-1}$, on déduit que $(v_i)$ est une
suite croissante qui tend vers l'infini. Ainsi $P_{i+1} - P_i$ converge
vers $0$ pour la topologie $u$-adique, et la suite des $(P_i)$ converge
vers une matrice $P$. Celle-ci vérifie $P G \phi(P)^{-1} = H$ et est
congrue à l'identité modulo $u$ du fait que chacun des $v_i$ est
strictement positif. Elle est donc inversible dans $k[[u]]$, comme
demandé.
\end{proof}

Le lemme nous assure que, quitte à modifier la base $(e_1, \ldots,
e_d)$, on peut remplacer $G$ par une matrice qui lui est congrue modulo
$u^N$. En particulier, on peut supposer que $G$ a tous ses coefficients
dans $\F_q[[u]]$ pour un certain $q$. C'est ce que nous ferons par la
suite.

\medskip

Soit $M$ le sous-$\F_q[[u]]$-module de $D$ engendré par les $e_i$ ; il
est libre de rang $d$. On définit deux suites récurrentes $(x_i)$ et
$(n_i)$ comme suit. On pose en premier lieu $x_0 = e_1$. Maintenant, si
$x_i$ est construit, on définit $n_i$ comme le plus petit entier tel que
$\phi(x_i) \in u^{n_i} M$ et $x_{i+1} = u^{-n_i} x_i$. On remarque tout
de suite que tous les $x_i$ sont des éléments de $M$ qui ne sont pas
dans $uM$.

\begin{lemme}
\label{lem:majni}
Pour tout $i$, on a $n_i \leq \gamma$.
\end{lemme}

\begin{proof}
Il suffit de montrer que si $x \in M$, $x \not\in uM$ alors $\phi(x)
\not\in u^{\gamma+1} M$. Or, le prémisse entraîne l'existence d'une
$k[[u]]$-base $(x_1, \ldots, x_d)$ de $M$ avec $x_1 = x$. Si $P$ est la
matrice de passage de $(e_1, \ldots, e_d)$ à $(x_1, \ldots, x_d)$ la
matrice du $\phi$ dans la base $(x_1, \ldots, x_d)$ est donnée par la
formule $H = P^{-1} G \phi(P)$. Comme $P$ est inversible dans $k[[u]]$,
son déterminant a une valuation $u$-adique nulle, d'où on déduit que le
déterminant de $H$ a pour valuation $u$-adique $\gamma$. Ainsi, sa première
colonne ne peut pas être multiple de $u^{\gamma+1}$ ce qui correspond 
exactement à ce que l'on voulait.
\end{proof}

Soit $c$ un entier strictement supérieur à $\frac \gamma {b-1}$. Notons
$\bar x_i$ la réduction modulo $u^c$ de $x_i$ : c'est un élément de
l'ensemble \emph{fini} $M/u^c M$. D'après le principe des tiroirs, il
existe deux indices $i < j$ tels que $\bar x_i = \bar x_j$. Posons
$\delta = j - i$ et $x = x_i$. En déroulant les définitions, on 
obtient :
$$\phi^\delta (x) \equiv u^n x \pmod {u^{n+c} M}$$
avec :
$$n = b^{\delta-1} n_i + b^{\delta-2} n_{i+1} + \cdots + b n_{j-2}
+ n_{j-1}.$$
Nous souhaitons à présent relever la dernière congruence en une vraie
égalité dans $M$. Pour cela, on commence par écrire $\phi^\delta(x) = 
u^n (x + u^c y)$ et on définit une nouvelle suite récurrente 
$(z_i)$ par $z_0 = x$ et $z_{i+1} = u^{-n} \phi^\delta(z_i)$. On a
$z_1 - z_0 = u^c y$ et $z_{i+1} - z_i = u^{-n} \phi^\delta(z_i - 
z_{i-1})$. Il s'ensuit que $z_{i+1} - z_i$ est un multiple de 
$u^{v_i}$ où $(v_i)$ est la suite récurrente définie par $v_0 = c$
et $v_{i+1} = b^\delta v_i - n$. Maintenant, le lemme \ref{lem:majni}
donne :
$$n \leq \gamma (b^{\delta-1} + b^{\delta-2} + \cdots + 1) = \gamma
\: \frac{b^\delta - 1} {b-1} < c (b^\delta - 1)$$
à partir de quoi on déduit $\lim_{i \to \infty} v_i = +\infty$.
Ainsi la suite $(z_i)$ converge vers un élément $z \in M$ (car tous les
$v_i$ sont positifs) vérifiant $\phi^\delta(z) = z$. Par ailleurs, on a
$z \equiv x \pmod u$, ce qui assure qu'il est non nul. On a donc montré
le résultat intermédiaire important suivant :

\begin{theo}
\label{theo:vp}
Soit $D$ un objet simple de $\Mp$. Alors, il existe des entiers $\delta 
> 0$, $n \geq 0$ et un élément non nul $z \in D$ tels que $\phi^\delta(z) 
= u^n z$.
\end{theo}

\begin{cor}
Soit $D$ un objet simple de $\Mp$.
\begin{itemize}
\item Si $\sigma$ n'est pas l'identité, il existe $r \in \calR_b$
tel que $D \simeq D(r)$.
\item Si $\sigma$ est l'identité, il existe $r \in \calR_b$ et 
$a \in k^\star$ tels que $D \simeq D(r,a)$.
\end{itemize}
\end{cor}

\begin{proof}
D'après le théorème \ref{theo:vp}, il existe des entiers $\delta > 0$,
$n \geq 0$ et un morphisme (dans $\Mp$) non nul $f : D(\delta, n) \to 
D$. La simplicité de $D$ assure que $f$ est surjectif, et donc que $D$
se retrouve parmi les constituants de Jordan-Hölder de $D(\delta,n)$.
Écrivons la fraction $r = \frac n {b^\delta -1}$ sous la forme $\frac m
{b^{\ell(r)} - 1}$. Le quotient $\frac {\delta} {\ell(r)}$ est alors un
nombre entier.

Si $\sigma$ n'est pas l'identité, la proposition \ref{prop:decomp}.(i)
montre que $D(\delta,n)$ s'écrit comme une somme directe de copies de
$D(r)$. En particulier, puisque les $D(r)$ sont simples d'après le
théorème \ref{theo:Dr}, tous les quotients de Jordan-Hölder de
$D(\delta,n)$ sont isomorphes à $D(r)$, et le théorème est démontré dans
ce cas.
Supposons maintenant qu'au contraire $\sigma = \id$. On écrit $\frac
{\delta} {\ell(r)} = p^v t$ où $t$ est un entier premier à $p$.
Plusieurs applications successibles de la proposition
\ref{prop:decomp}.(iii) montrent que $D(\delta,n)$ admet une suite de
composition dont les quotients successifs sont tous isomorphes à
$D(\delta p^{-v}, n')$ pour un certain entier $n'$. L'alinéa (ii) de la
même proposition montre alors que les constituants de Jordan-Hölder sont
dans ce cas tous isomorphes à des $D(r,a)$ pour certains éléments $a$ de
$k^\star$ (qui peuvent varier d'un composant à l'autre). La conclusion
en découle.
\end{proof}

\end{document}